\documentclass[11pt]{article} 
\usepackage{amsmath}
\usepackage{amssymb}
\usepackage{amsthm,mathtools}
\usepackage{lineno}
\usepackage[mathscr]{eucal}
\usepackage{xcolor}
\usepackage{fontenc}
\usepackage{graphicx}
\usepackage{geometry}
\usepackage{lineno}

\theoremstyle{definition}
\newtheorem{definition}{Definition}[section]
\newtheorem{theorem}[definition]{Theorem}
\newtheorem*{theorem*}{Conjecture}
\newtheorem{proposition}[definition]{Proposition}
\newtheorem{lemma}[definition]{Lemma}
\newtheorem{corollary}[definition]{Corollary}
\theoremstyle{remark}
\newtheorem{remark}[definition]{Remark}
\newtheorem{example}[definition]{Example}
\newcounter{enumctr}

\newcommand{\R}{\mathbb{R}}
\newcommand{\N}{\mathbb{N}}

\def\du{{\ensuremath{\mathrm{d}}}}

\providecommand{\keywords}[1]{\textbf{\textbf{Key words: }} #1}
\begin{document}
	\title{Asymptotic behavior of solutions to some classes of multi-order fractional cooperative systems} 
	\author{La Van Thinh\footnote{\tt lavanthinh@hvtc.edu.vn, \rm Academy of Finance, No. 58, Le Van Hien St., Duc Thang Wrd., Bac Tu Liem Dist., Hanoi, Viet Nam},  Hoang The Tuan\footnote{\tt tuanht@gbu.edu.vn, \rm Department of Mathematics, Great Bay University, Dongguan, Guangdong 523000, China and Institute of Mathematics, Vietnam Academy of Science and Technology, 18 Hoang Quoc Viet, 10307 Ha Noi, Viet Nam}}
	\date{18/11/2023}
	\maketitle
\begin{abstract}
This paper is devoted to the study of the asymptotic behavior of solutions to multi-order fractional cooperative systems. First, we demonstrate the boundedness of solutions to fractional-order systems under certain conditions imposed on the vector field.  We then prove the global attractivity and the convergence rate of solutions to such systems (in the case when the orders of fractional derivatives are equal, the convergence rate of solutions is sharp and optimal). To our knowledge, these kinds of results are new contributions to the qualitative theory of multi-order fractional positive systems and they seem to have been unknown before in the literature. As a consequence of this result, we obtain the convergence of solutions toward a non-trivial equilibrium point in an ecosystem model (a particular class of fractional-order Kolmogorov systems). Finally, some numerical examples are also provided to illustrate the obtained theoretical results.
\end{abstract}
\\

\section{Introduction}
Positive systems are dynamic systems in which their state variables remain in the first orthant of $\R^d$ when the initial conditions are initiated in this domain. Up to now, an impressive number of theoretical and
applicative contributions to this theory have been published, see, e.g., 
\cite{Nieuwenhuis,Coxson,Carson,Benvenuti,Moreno,Haddad,Vargas,Vecchio,Blanchini}.

A special class of nonlinear positive systems is the cooperative systems which have been discussed extensively, especially in connection with biological applications, see, e.g., {\cite{Hirsch,Smillie,Smith86,H.L.Smith,Shen}} while the cooperative systems with the added homogeneous structure are mentioned in \cite{Aeyels2,FCM}. In particular, consider the system 
\begin{align}
\frac{d}{dt}x(t)&=f(x(t)),\;t>0,\label{e01}\\
x(0)&=x^0\in\R^d_{\geq 0},\label{e02}
\end{align}
here the vector field $f(\cdot)$ is {\em homogeneous of degree $p\geq 1$} and {\em cooperative}. From the perspective of positive system theory, in \cite{Aeyels1}, the authors have proven that the system \eqref{e01}--\eqref{e02} is {\em asymptotically stable} if and only if there exists a vector $v\succ 0$ such that $f(v)\prec 0$. When $f(\cdot)$ is {\em homogeneous}, this result is extended to arbitrary initial conditions $x^0\in \R^d$ by O. Mason and M. Verwoerd \cite{Mason}.

Due to the usefulness of fractional calculus compared to classical analysis in modelling many processes that emerged from different fields of science and engineering (see, e.g., \cite{Baleanu_1,Baleanu_2,Petras,Tarasov_1,Tarasov_2}), our aim in the present work is to study the asymptotic behaviour of solutions to fractional-order systems where homogeneous and cooperative assumptions are satisfied. We note that in this case, the existence and uniqueness of solutions have not been investigated in the literature. On the other hand, the approaches using the comparison principle based on the geometric interpretation of the classical derivative and the local nature of solutions as in the two papers mentioned above do not seem to be applicable.

The article is organized as follows. Notation and some mathematical background are introduced in Section 2. The main content of the paper is presented in Section 3. In particular, in this part, we first show the boundedness of solutions to some classes of multi-order fractional cooperative systems. After that, we prove the global attractivity and the convergence rate of solutions to such systems. As a consequence, we study an ecosystem model (fractional-order Lotka--Volterra type systems) and describe the convergence of solutions toward its non-trivial equilibrium point. Finally, numerical examples are provided in Section 4 to illustrate the proposed theoretical results.
\section{Notation and preliminaries}
\subsection{Notation}
In this paper, we use the following notations: $\N,\ \R$ are the sets of natural numbers, real
numbers, respectively; $\R_{\geq 0}:=\{x\in \R:x\geq 0\}$, $\R_+:=\{x\in \R:x> 0\}$; $\R^d$ stands for the $d$-dimensional
Euclidean space; $\R_{\geq 0}^d$, $\R_+^d$ are the subsets of $\R^d$ with
nonnegative entries and positive entries, respectively.
Let $x,y\in\R^d$, then $[x;y]:=\{s\in\R^d: s=tx+(1-t)y,t\in [0,1]\}$. For two vectors $w,u\in \R^d$, we write
\begin{itemize}
	\item $u\succeq w$ if $u_i\geq w_i$ for all $1\leq i\leq d.$  
	
	\item $u\succ w$ if $u_i> w_i$ for all $1\leq i\leq d.$
\end{itemize}
Let $r>0$, we set $B_r(0):=\{x\in \R^d:\|x\|\leq r\}$ and  $\partial B_r(0):=\{x\in \R^d:\|x\|= r\}$. For a vector-valued function  $f: \R^d \longrightarrow \R^d$ which is differentiable at $x\in \R^d$, we denote $Df(x):=\big(\frac{\partial f_i}{\partial x_j}(x)\big)_{1\leq i,j\leq d}$.
Fixing a vector $v\succ 0$, the weighted  norm $\|.\|_v$ on $\R^d$ is defined by $\|w\|_v:=\max_{1\le i \le d} \frac{|w_i|}{v_i}.$
A real matrix $A=(a_{ij})_{1\leq i,j\leq d}$ is Metzler if its off-diagonal entries $a_{ij},\ \forall i\ne j$, are nonnegative. 

Let $\alpha \in (0,1]$ and $J = [0, T]$, the Riemann-Liouville fractional integral of a function $x :J \rightarrow \mathbb R$ is denoted by 
	$$ 
		I^\alpha_{0^+}x(t) := \frac{1}{\Gamma(\alpha)}\int_{0}^{t}(t-s)^{\alpha -1}x(s) \, \du s,\ \quad t\in J,
	$$
	and the Caputo fractional derivative of the order $\alpha$ is given by
	$$ 
		^C D^\alpha_{0^+}x(t) := \frac{d}{d t}I^{1-\alpha}_{0^+}(x(t) - x(0)), \quad t \in J \setminus \{ 0 \},
	$$
	here $\Gamma(\cdot)$ is the Gamma function and $\displaystyle\frac{\du}{\du t}$ is the classical derivative (see, e.g., 
	\cite[Chapters 2 and 3]{Kai} and \cite{Vainikko_16} for more detail on fractional calculus). For $d\in\N,\ \hat{\alpha}:=\left(\alpha_1,\dots, \alpha_d\right)\in (0,1]^d$ and a function $w : J \rightarrow \R^d,$ we use the notation
$$^{\!C}D^{\hat{\alpha}}_{0^+}w(t):=\left(^{\!C}D^{\alpha_1}_{0^+}w_1(t),\dots,^{\!C}D^{\alpha_d}_{0^+}w_d(t)\right)^{\rm T}.$$
\begin{definition} \cite[Definition 2.3]{Mason}
	A vector field $f: \R^d \longrightarrow \R^d$ is said to be homogeneous
	 if for all $x\in \R^d$ and for all $\lambda >0$, we have
	$$f(\lambda x)=\lambda f(x).$$
\end{definition}
\begin{definition} \cite[Definition 3]{Qiang-xiao}
	A vector field $f: \R^d \longrightarrow \R^d$ is called homogeneous
	of degree $p>0$ if for all $x\in \R^d, \lambda >0$ we have
	$$f(\lambda x)=\lambda^pf(x).$$
\end{definition}
\begin{definition} \cite[Definition 2]{Qiang-xiao}
	A continuous vector field $f: \R^d \longrightarrow \R^d$ which is
	continuously differentiable on $\R^d\backslash \left\{0\right\}$ is said to be cooperative if the
	Jacobian matrix $Df(x)$ is Metzler for all $x\in \R_{\geq 0}^d\backslash \left\{0\right\}$.
\end{definition}
Let $\hat{\alpha}=(\alpha_1,\dots,\alpha_d)^{\rm T}\in (0,1]^d$. Our main object in the paper is the fractional-order nonlinear system
\begin{equation} \label{Eq main}
	\begin{cases}
		^{\!C}D^{\hat{\alpha}}_{0^+}w(t)&=f(w(t)),\ \forall t> 0,\\
		w(0)&=\omega \in \R^d_{\geq 0},
	\end{cases}
\end{equation}
where $f = (f_1,\dots, f_d)^{\rm T}$ with $f_i:\R^d \rightarrow \R ,\ i = 1,\dots, d,$ satisfies some following assumptions.
\begin{itemize}
		\item[(A1)] $f(\cdot)$  is cooperative.
		\item[(A2)] $f(\cdot)$ is homogeneous of degree $p\geq 1$.
\item[(A3)] There exists $v\succ 0$ such that $f(v)\prec 0$.
	\end{itemize}
Following from Proposition \ref{glcd1} and Proposition \ref{loclc} below, for each $\omega\in\R^d_{\geq 0}$, the system \eqref{Eq main} has a unique solution $\Phi(\cdot,\omega)$ on the maximal interval of existence $[0, T_{\max}(\omega))$. 
\begin{definition}
	System \eqref{Eq main} is strictly monotone if for any $\lambda^1,\lambda^2\in \R^d_{+}$, $\lambda^1\prec \lambda^2$, we have
\[
\Phi(t,\lambda^1)\prec \Phi(t,\lambda^2),\;\forall t\in (0,T_{\max}(\lambda^1))\cap (0,T_{\max}(\lambda^2)).
\] 
\end{definition}
\begin{definition}
	System \eqref{Eq main} is monotone if for any $\lambda^1,\lambda^2\in \R^d_{\geq 0}$, $\lambda^1\preceq \lambda^2$, we have
\[
\Phi(t,\lambda^1)\preceq \Phi(t,\lambda^2),\;\forall t\in (0,T_{\max}(\lambda^1))\cap (0,T_{\max}(\lambda^2)).
\] 
\end{definition}
\begin{definition}
	System \eqref{Eq main} is positive if for any $\omega \succeq 0$, its solution $\Phi(\cdot, \omega)$ satisfies
	\begin{align*}
		\Phi(\cdot, \omega)\succeq 0\ \text{on}\ [0, T_{\max}(\omega)).
	\end{align*}
\end{definition}

\subsection{Preliminaries}
We collect here some preparatory knowledge that plays an essential role for further analysis in the rest of the paper.
\begin{proposition}\cite[Lemma 2.1]{Mason}\label{glcd1}
Suppose that $f: \R^d \longrightarrow \R^d$ is continuous and is
	continuously differentiable on $\R^d\backslash \left\{0\right\}$. Moreover, this function is homogeneous. Then, there exists a positive constant $K$ such that $\|f(x)-f(y)\|\leq K\|x-y\|$, $\forall x,y\in \R^d$. 
\end{proposition}
\begin{proposition}\label{loclc}
Suppose that $f: \R^d \longrightarrow \R^d$ is continuous and is continuously differentiable on $\R^d\backslash \left\{0\right\}$. In addition, we assume that $f$ is homogeneous of degree $p>1$. Then, for any $r>0$, we can find a positive constant $K$ that depends on $r$ satisfying $\|f(x)-f(y)\|\leq K\|x-y\|$, $\forall x,y\in B_r(0)$. In particular, $f$ is Lipschitz continuous on balls centered at the origin and with arbitrary radius. 
\end{proposition}
\begin{proof}
Due to the fact that $f$ is continuously differentiable on $\R^d\backslash \left\{0\right\}$, we have $$K_1:=\sup_{x\in \partial B_1(0)}\|Df(x)\|<\infty.$$ Furthermore, based on the assumption that $f$ is homogeneous of degree $p>1$ on $\R^d$, we see that $Df(\lambda x)=\lambda^{p-1}Df(x)$ for all $x\in\R^d\backslash \left\{0\right\}$ and $\lambda>0$. Hence,
\begin{align}\label{glc}
\notag\|Df(x)\|&=\|x\|^{p-1}\|Df(\frac{x}{\|x\|})\|\\
&\leq K_1\|x\|^{p-1},\; \forall x\in \R^d\backslash \left\{0\right\}.
\end{align}
Choose any $x\in B_1(0)\backslash \left\{0\right\}$ and then fix it, by the mean value theorem, we obtain the following estimate
\begin{equation*}
\|f(x)-f(y)\|\leq \|Df(\theta)\|\|x-y\|,\;\forall y\in B_1(0)\setminus\{tx:t\leq 0\},
\end{equation*}
where $\theta \in [x;y]$, which together with \eqref{glc} implies that
\begin{align}\label{glc1}
\notag\|f(x)-f(y)\|&\leq \|Df(\theta)\|\|x-y\|\\
\notag&\leq K_1\|\theta\|^{p-1}\|x-y\|\\
&\leq K_1\|x-y\|,\;\forall y\in B_1(0)\setminus\{tx:t\leq 0\}.
\end{align}
However, from the continuity of $f(\cdot)$ on $\R^d$, it follows that the inequality \eqref{glc1} is true for any $y\in B_1(0)$. Notice that $x$ is arbitrarily in $B_1(0)\backslash \left\{0\right\}$, thus this estimate holds for every $y\in B_1(0)$, $x\in B_1(0)\backslash \left\{0\right\}$. Using the continuity of the function $f(\cdot)$ again, we get \eqref{glc1} for all $x,y\in B_1(0)$. This means that
\begin{equation}\label{add_r1_1}
  \|f(x)-f(y)\|\leq K_1\|x-y\|,\;\forall x,y\in B_1(0).  
\end{equation}
We now consider the case $x,y\in B_r(0)$ with $r>1$. There are four cases: I. $x,y\in B_r(0)\setminus B_1(0)$; II. $x \in B_r(0)\setminus B_1(0)$ and $y\in B_1(0)$; III. $y\in B_r(0)\setminus B_1(0)$ and $x\in B_1(0)$; IV. $x,y\in B_1(0)$. The estimate for case IV is shown above. For case I, if $[x;y]\cap \partial B_1(0)=\emptyset$, then
\begin{equation}\label{glc2}
\|f(x)-f(y)\|\leq K_2\|x-y\|,
\end{equation}
where $K_2:=\sup_{x\in B_r(0)\setminus B_1(0)}\|Df(x)\|<\infty$. Notice that the estimate \eqref{glc2} is also true for $x\in B_r(0)\setminus B_1(0)$, $y\in \partial B_1(0)$ or $y\in B_r(0)\setminus B_1(0)$, $x\in \partial B_1(0)$. Suppose that $[x;y]\cap \partial B_1(0)=\{x_1,y_1\}$. Then,
\begin{align}\label{glc3}
\|f(x)-f(y)\|&=\|f(x)-f(x_1)+f(x_1)-f(y_1)+f(y_1)-f(y)\|\notag\\
&\leq K_2\|x-x_1\|+K_1\|x_1-y_1\|+K_2\|y_1-y\|\notag\\
&\leq K(\|x-x_1\|+\|x_1-y_1\|+\|y_1-y\|)\notag\\
&=K\|x-y\|,
\end{align}
where $K:=\max\{K_1,K_2\}$. For case II, let $\{x_1\}=[x;y]\cap \partial B_1(0)$. It is easy to see
\begin{align}\label{glc4}
\|f(x)-f(y)\|&=\|f(x)-f(x_1)+f(x_1)-f(y)\|\notag\\
&\leq K_2\|x-x_1\|+K_1\|x_1-y\|\notag\\
&\leq K(\|x-x_1\|+\|x_1-y\|)\notag\\
&=K\|x-y\|.
\end{align}
By the same arguments as in the proof of case II, for case III, we also have
\begin{equation*}
\|f(x)-f(y)\|\leq K\|x-y\|.
\end{equation*}
In short, based on the obtained observations \eqref{add_r1_1}, \eqref{glc2}, \eqref{glc3} and \eqref{glc4}, for any $r>0$, we have proved that $\|f(x)-f(y)\|\leq K \|x-y\|$ for all $x,y\in B_r(0)$, where the positive constant $K$ depends on $r$. The proof is complete.
\end{proof}
\begin{proposition} {\cite[Remark 1.1, Chapter 3, p. 33]{H.L.Smith}} \label{Pro_Cooperative}
	Let $f: \R^d \longrightarrow \R^d$ be a
	cooperative vector field. For any two vectors $u,w\in \R_{\geq 0}^d$ with $u_i=w_i,\; i\in\{1,\cdots,d\}$
	and $u \succeq w$, we have $$f_i(u)\ge f_i(w).$$
\end{proposition}
\begin{lemma}\label{compare-FDE}
	Let $w : [0, T] \rightarrow \R$ be continuous and assume that the Caputo derivative
	$^{\!C}D^{\alpha}_{0^+}w(\cdot)$ is also continuous on the interval $[0, T]$ with $\alpha\in (0,1]$. If there exists $t_0>0$ such that $w(t_0)=0$ and $w(t)<0,\ \forall t\in [0,t_0)$, then 
 \begin{itemize}
     \item[(i)] $^{\!C}D^{\alpha}_{0^+}w(t_0)> 0$ for $0<\alpha<1$;
     \item[(ii)] $^{\!C}D^{\alpha}_{0^+}w(t_0)\geq 0$ for $\alpha=1$.
 \end{itemize}
\end{lemma}
\begin{proof}
The conclusion of the case $\textup{(ii)}$ is obvious. The proof of the case $\textup{(i)}$ follows directly from \cite[Theorem 1]{Vainikko_16}.
\end{proof}
\begin{remark}
A weaker version of Lemma \ref{compare-FDE} was introduced in \cite[Lemma 25]{Tuan}.  
\end{remark}
\section{Asymptotic behavior of solutions to fractional-order cooperative systems}
This section represents our main contributions. First, we show the boundedness of solutions to multi-order fractional cooperative homogeneous systems. We then prove the global attractivity and the convergence rate of solutions to such systems. Finally, we obtain the convergence of solutions toward a non-trivial equilibrium point of a fractional-order Lotka-Volterra type model.
\subsection{Boundedness and positivity of solutions to cooperative systems}
\begin{proposition} \label{boundsol}
Consider the system \eqref{Eq main}. Suppose that $f(\cdot)$ satisfies the assumptions $\textup{(A1)}$, $\textup{(A2)}$. In addition, there exists a vector $v \succ 0$ such that $\textup{(A3)}$ is true. Then, for any $\omega\succ 0$, the solution $\Phi(\cdot,\omega)$ exists on $[0,\infty)$. Moreover, we have
\[\|\Phi(t,\omega)\|_v\leq \|\omega\|_v,\;\forall t\geq 0.\]
\end{proposition}
\begin{proof}
{\bf The case: $p=1$.} Based on Proposition \ref{glcd1}, the vector field $f(\cdot)$ is global Lipschitz continuous on $\R^d$. It leads to that,
for every $\omega\succ 0$, the system \eqref{Eq main} has the unique global solution $\Phi(t,\omega)$ on $[0,\infty)$. 
Let ${\epsilon} >0$ be arbitrary. For each $i=1,\dots,d$, we define $$y_i(t):=\frac{{\Phi}_i(t,\omega)}{v_i} - \|\omega\|_v-{\epsilon},\ \forall t\ge0.$$
	Notice that $$y_i(0)=\frac{w_i}{v_i} - \|\omega\|_v-{\epsilon}<0,\ \forall i=\overline{1,d}.$$
Thus, if there is a $t>0$ and an index $i$ with  $y_i(t)=0$, by choosing 
	\begin{align*} 
		t_*:=\inf\{t>0:\exists i= \overline{1,d}\ \text{such that}\ y_{i}(t)=0\},
	\end{align*}  
	then $t_*>0$ and there exists an index $ i^*$ which verify
	\begin{align} \label{t_*^}
		y_{i^*}(t_*)&=0\ \text{and}\ y_{i}(t_*)\le0,\ \forall i\ne i^*, \\
		y_{i^*}(t)&<0,\ \forall t\in[0,t_*). \notag
	\end{align}
This implies that 
\begin{align} 
		{\Phi}_{i^*}(t_*,\omega)&=(\|\omega\|_v+{\epsilon})v_{i^*},\;\Phi_{i^*}(t,\omega)<(\|\omega\|_v+{\epsilon})v_{i^*},\;\forall t\in [0,t_*),\label{estbd1}\\
	{\Phi}_{i}(t_*,\omega)&\le(\|\omega\|_v+{\epsilon})v_{i},\ \forall i\ne i^*.\label{estbd2}
	\end{align}
By combining \eqref{t_*^} and Lemma \ref{compare-FDE}, we obtain \begin{equation}\label{contrabd}
^{C\!}D^{\alpha_{i^*}}_{0^+}y_{i^*}(t_*)\geq 0.\end{equation}
On the other hand, following from \eqref{estbd1}, \eqref{estbd2} and Proposition \ref{Pro_Cooperative}, we observe that
	\begin{align*}
		^{C\!}D^{{\alpha_{i^*}}}_{0^+}y_{i^*}(t_*)&=\frac{^{C\!}D^{\alpha_{i^*}}_{0^+}\Phi_{i^*}(t_*,\omega)}{v_{i^*}} \\
		&=\frac{1}{v_{i^*}}f_{i^*}(\Phi(t_*,\omega)) \\
		&\le\frac{1}{v_{i^*}}f_{i^*}\left((\|\omega\|_v+{\epsilon})v\right)\\&=(\|\omega\|_v+{\epsilon})\frac{f_{i^*}(v)}{v_{i^*}} <0,
	\end{align*}
which contradicts \eqref{contrabd}. This means that $y_i(t)<0$ all $t\geq 0$ and for all $i=1,\dots,d$. Hence,
$$\frac{{\Phi}_i(t,\omega)}{v_i} < \|\omega\|_v+{\epsilon},\ \forall t\ge0,\ i=1,\dots,d.$$
Let ${\epsilon}\to 0$, we have 
$$\frac{{\Phi}_i(\cdot,\omega)}{v_i} \le \|\omega\|_v,\ \forall t\ge0,\ i=1,\dots,d.$$ The desired estimate is checked.

\noindent {\bf The case: $p>1$}. Under Proposition \ref{loclc}, the vector-valued function $f(\cdot)$ is Lipschitz continuous on $B_r(0)$ for any $r>0$. 
Thus, for any initial condition $\omega  \succ 0$, 
the system  \eqref{Eq main} has a unique solution $\Phi(\cdot,\omega)$ on the maximal interval of existence $[0,T_{\max}(\omega))$. 
Now, by using the same arguments as in the proof of the case $p=1$, it is not difficult to show that
\begin{equation}\label{tam}
\|\Phi(t,\omega)\|_v\leq \|\omega\|_v,\;\forall t\in [0,T_{\max}(\omega)).
\end{equation}
However, in light of \eqref{tam} and the definition of the maximal interval of existence, it must be true that $T_{\max}(\omega)=\infty$ because otherwise the solution $\Phi(\cdot,\omega)$ can be extended over a larger interval. The proof of the theorem is complete.
\end{proof}
\begin{lemma}\label{Sys positive}
Consider the system \eqref{Eq main}. Suppose that the assumptions $\textup{(A1)}$, $\textup{(A2)}$ and $\textup{(A3)}$ are satisfied. Then, the system \eqref{Eq main} is positive.
\end{lemma}
\begin{proof} Take and fix the initial condition $\omega\succeq 0$. Let $\Phi^n(\cdot,\omega^n)$ be the unique solution of the system
\begin{equation} \label{Eq main_n}
	\begin{cases}
		^{\!C}D^{\hat{\alpha}}_{0^+}x(t)&=f(x(t))+\displaystyle\frac{\textbf e}{n},\ \forall t>0,\\
		x(0)&=\omega^n,
	\end{cases}
\end{equation}
where $\omega^n=\omega+\displaystyle\frac{1}{n}\textbf{e}$ and  $\textbf{e}:=(1,\dots,1)^{\rm T} \in \R^d$, $n\in \N$. For each $n\in \mathbb N$, it follows from Proposition \ref{boundsol} that $\Phi^n(t,\omega^n)\succ 0$ for all $t\geq 0$.
Let $m,n\in \mathbb N$, $m>n$ and put $\Psi(t):=\Phi^m(t,\omega^m)-\Phi^n(t,\omega^n),\ \forall t\in [0, \infty).$ We first show that $\Psi(t)\prec 0$ for all $t\geq 0$. Indeed, if this statement is false, there exists a $t\in(0,\infty)$ and an index $i=1,\dots,d$ with $\Psi_{i}(t)=0$. Take
	\begin{align*}
		t_*:=\inf \{t>0:\exists i= \overline{1,d}\ \text{such that}\ \Psi_{i}(t)=0\}.
	\end{align*}  
Then, $t_*>0$ and there is an index $ i_*$ such that  
	\begin{align} \label{t_*}
		\Psi_{i_*}(t_*)&=0,\quad \Psi_{i}(t_*)\leq 0,\ i\ne i_*,\\
		\quad \Psi_{i}(t)&<0,\ \forall t\in[0,t_*),\ i=1,\dots,d. \notag
	\end{align}
Since $\Psi_{i_*}(t_*)=0$ and $\Psi_{i_*}(t)<0,\ \forall t\in [0,t_*)$, by Lemma \ref{compare-FDE}, it deduces that  $$^{\!C}D^{\alpha_{i_*}}_{0^+}\Psi_{i_*}(t_*)\geq 0.$$ 
On the other hand, from \eqref{t_*}, we have
	\begin{align*}
		\Phi_{i_*}^m(t_*,\omega^m)&=\Phi_{i_*}^n(t_*,\omega^n),\\
		\Phi_{i}^m(t_*,\omega^m)&\leq \Phi_{i}^n(t_*,\omega^n),\ \forall i\ne i_*,
	\end{align*}
	which together with Proposition \ref{Pro_Cooperative} implies $f_{i_*}(\Phi^m(t_*,\omega^m))\leq f_{i_*}(\Phi^n(t_*,\omega^n))$.  This leads to that 
	\begin{align*}
		^{\!C}D^{\alpha_{i_*}}_{0^+}\Psi_{i_*}(t_*)&={^{\!C}D^{\alpha_{i_*}}_{0^+}}\Phi_{i_*}^m(t_*,\omega^m)-	{^{\!C}D^{\alpha_{i_*}}}_{0^+}\Phi_{i_*}^n(t_*,\omega^n)\\
		&=f_{i_*}(\Phi^m(t_*,\omega^m))+\frac{1}{m}-f_{i_*}(\Phi^n(t_*,\omega^n))-\frac{1}{n}\\
		&< 0,
	\end{align*}
	a contradiction. This means that the sequence $\left\{\Phi^n(\cdot,\omega^n)\right\}_{n=1}^{\infty}$ is positive, strictly decreasing, continuous on $[0,\infty)$. Thus, for each $t\ge0$, the limit below exists 
$$\Psi^*(t):=\displaystyle\lim_{n\to\infty}\Phi^n(t,\omega^n).$$
It is clear to see that  $\left\{\Phi^n(\cdot,\omega^n)\right\}_{n=1}^{\infty}$
converges uniformly to $\Psi^*(\cdot)$ and $\Psi^*(\cdot)$ is also continuous and nonnegative on each interval $[0,T]$ with $T>0$ is arbitrary. On the other hand, for each $n \in \N$, we observe 
\begin{align*}
	\Phi_i^n(t,\omega^n)=\omega_i+\frac{1}{n}+\frac{t^{\alpha_i}}{n\Gamma(\alpha_i+1)}+\frac{1}{\Gamma(\alpha_i)}\int_0^t(t-s)^{\alpha_i-1}f_i(\Phi^n(s,\omega^n))ds
\end{align*}
with $t\in [0,\infty)$, $i=1,\dots,d$. For each $t\geq 0$, let  $n \rightarrow \infty$, we conclude 
\begin{align*}
	\Psi_i^*(t)=\omega_i+\frac{1}{\Gamma(\alpha_i)}\int_0^t(t-s)^{\alpha_i-1}f_i(\Psi^*(s))ds,\;\forall t\geq 0,\;i=1,\dots,d.
\end{align*} 
This together with the fact the system \eqref{Eq main} has a unique solution on $[0,\infty)$ deduces that $\Psi^*(t) = \Phi(t,\omega),\ \forall t\in[0,\infty)$. In particular, we have shown that $\Phi(t,\omega)\succeq 0$ for all $t\geq 0$ which finishes the proof.
\end{proof}
\begin{remark}
From the proof of Lemma \ref{Sys positive}, it is easy to see that the system \eqref{Eq main} is monotone.
\end{remark}
\begin{remark}\label{add_r_1}
The conclusion of Proposition \ref{boundsol} is still true when the initial condition $\omega\in\R^d_{\geq 0}$. 
\end{remark}
\subsection{Global attractivity and convergence rate of solutions to cooperative systems}
In this section, we discuss the attractivity and the convergence rate of solutions to the system \eqref{Eq main}
	\begin{equation*} 
		\begin{cases}
			^{\!C}D^{\hat{\alpha}}_{0^+}w(t)&=f(w(t)),\ \forall t> 0,\\
			w(0)&=\omega \in \R^d_{\geq 0}.
		\end{cases}
	\end{equation*}
	\begin{theorem} \label{Rate-p>1}
	Suppose that $f(\cdot)$ satisfies the assumptions $\textup{(A1)}$, $\textup{(A2)}$. If there exists $v\succ0$ so that $f(v)\prec0$, then for each $\omega\in\R^d_{\geq 0}$, we can find constants $\eta>0,\,C>0$ such that	
	\begin{equation}\label{ctest2}
		0\le \Phi_i(t,\omega)\leq  C E_{\underline{\alpha}/p}(-\eta t^{\underline{\alpha}/p}),\;\forall t\ge 0,\ i=1,\dots,d,
	\end{equation}
	where $\underline{\alpha}:=\displaystyle\min_{1\leq i\leq d}\alpha_i$.
	\end{theorem}
	\begin{proof}
	Let  $v\succ0$ which satisfies $f(v)\prec0$. For any initial condition $\omega\in\R^d_{\geq 0}$, by Proposition \ref{boundsol} and Remark \ref{add_r_1}, we see that the system \eqref{Eq main} has a unique global solution $\Phi(\cdot,\omega)$ with $\|\Phi(t,\omega)\|_v\leq \|\omega\|_v$ for all $t\geq 0$. We are only interested in the case when $\|\omega\|>0$. Let $m=\|\omega\|_v$ and choose $\eta>0$ satisfying
	\begin{equation}\label{addf}
	\frac{f_{i}(v)}{v_{i}}+\frac{\eta}{m^{p-1}}\sup_{t\geq 1}\frac{t^{\underline{\alpha}/p-\alpha_i}E_{\underline{\alpha}/p,1+\underline{\alpha}/p-\alpha_i}(-\eta t^{\underline{\alpha}/p})}{\bigg(E_{\underline{\alpha}/p}(-\eta t^{\underline{\alpha}/p})\bigg)^p}<0,\ \forall i=1,\dots,d.
 \end{equation}
Put
	\begin{equation*}
		u(t):={m_\varepsilon}E_\beta(-\eta t^\beta),\;t\geq 0,
	\end{equation*}
here $\beta>0$ will be chosen later and ${m_\varepsilon}:=\displaystyle\frac{m+\varepsilon}{E_\beta(-\eta)}$ with $\varepsilon>0$ is arbitrarily small. We will compare the solution of the system \eqref{Eq main} with the vector-valued function $u\bf{e}$.
By a direct computation, for any $\alpha\in (0,1)$, we have
\begin{align*}
\frac{1}{m_\varepsilon}{^{\!C}D^{{\alpha}}_{0^+}u(t)}&=\frac{1}{\Gamma (1-\alpha)}\int_0^t (t-s)^{-\alpha}\frac{d}{ds}E_\beta(-\eta s^\beta)ds\\
&=\frac{1}{\Gamma (1-\alpha)}\int_0^t (t-s)^{-\alpha}\sum_{k=1}^\infty \frac{(-\eta)^k k\beta s^{\beta k-1}}{\Gamma(\beta k+1)}ds\\
&=\frac{1}{\Gamma (1-\alpha)}\sum_{k=1}^\infty \frac{(-\eta)^k}{\Gamma(\beta k)}\int_0^t (t-s)^{-\alpha}s^{\beta k-1}ds\\
&=\frac{1}{\Gamma (1-\alpha)}\sum_{k=1}^\infty \frac{(-\eta)^kt^{-\alpha+\beta k}}{\Gamma(\beta k)}\int_0^1 \tau^{-\alpha}(1-\tau)^{\beta k-1}d\tau\\
&=\frac{1}{\Gamma (1-\alpha)}\sum_{k=1}^\infty \frac{(-\eta)^kt^{-\alpha+\beta k}}{\Gamma(\beta k)}B(1-\alpha,\beta k)\quad (\text{here}\;B(\cdot,\cdot)\;\text{is the Beta function})\\
&=-\eta\sum_{k=1}^\infty \frac{(-\eta)^{k-1}t^{-\alpha+\beta (k-1)+\beta}}{\Gamma(\beta (k-1)+1-\alpha+\beta)}\\
&=-\eta\sum_{k=0}^\infty \frac{(-\eta)^{k}t^{-\alpha+\beta k+\beta}}{\Gamma(\beta k+1-\alpha+\beta)}\\
&=-\eta t^{-\alpha+\beta}E_{\beta,1-\alpha+\beta}(-\eta t^\beta),\quad \forall t>0.
\end{align*}
The formula above is also true when $\alpha=1$. Define
	$$z_i(t):=\frac{\Phi_i(t,\omega)}{v_i}- u(t),\quad t\ge 0,\;i=1,\dots,d.$$
	Since $z_{i}(t)<0$ for all $t\in[0,1]$ and $i=1,\dots,d$, if the statement that $z(t)\preceq 0$ for all $t\geq 0$ is false, we can find an $t_*>1$ and an index $i_*$ so that
	\begin{align}\label{ctra2-2}
		z_{i_*}(t_*)&=0\ \text{and}\	z_i(t_*)\le0,\ \forall i\ne i_*,\\
		\notag z_{i_*}(t)&<0,\ \forall t\in [0,t_*).
	\end{align}
	This implies
	\begin{align*}
		\Phi_{i_*}(t_*,\omega)&= u(t_*)v_{i_*}\ \text{and}\	\Phi_i(t_*,\omega)\le u(t_*)v_{i},\ \forall i\ne i_*,\\
		\Phi_{i_*}(t,\omega)&<u(t)v_{i_*},\ \forall t\in [0,t_*).
	\end{align*}
	Due to the assumptions $\textup{(A1)}$,  $\textup{(A2)}$ and Proposition \ref{Pro_Cooperative}, the following estimate holds 
	$$f_{i_*}(\Phi(t_*,\omega))\le f_{i_*}\big(u(t_*)v\big)=f_{i_*}\big({m_\varepsilon}E_\beta(-\eta t_*^\beta) v\big)=\left({m_\varepsilon}E_\beta(-\eta t_*^\beta)\right)^pf_{i_*}(v). $$ 
	Then,
	\begin{align*}
		^{\!C}D^{{\alpha}_{i_*}}_{0^+}z_{i_*}(t_*)&=^{\!C}D^{\alpha_{i_*}}_{0^+}\frac{\Phi_{i_*}(t_*,\omega)}{v_{i_*}}-^{\!C}D^{{\alpha}_{i_*}}_{0^+}u(t_*)\\
		&=\frac{1}{v_{i_*}}f_{i_*}(\Phi(t_*,\omega))+{m_\varepsilon}\eta t^{\beta-\alpha_{i_*}}_*E_{\beta,1-\alpha_{i_*}+\beta}(-\eta t^\beta_*)\\
		&\le \left({m_\varepsilon}E_\beta(-\eta t_*^\beta)\right)^p \frac{f_{i_*}(v)}{v_{i_*}}+{m_\varepsilon}\eta t^{\beta-\alpha_{i_*}}_*E_{\beta,1-\alpha_{i_*}+\beta}(-\eta t^\beta_*)\\
		&=\left({m_\varepsilon}E_\beta(-\eta t_*^\beta)\right)^p \left[\frac{f_{i_*}(v)}{v_{i_*}}+\frac{\eta t^{\beta-\alpha_{i_*}}_*E_{\beta,1-\alpha_{i_*}+\beta}(-\eta t^\beta_*)}{(m_\varepsilon)^{p-1}\bigg(E_\beta(-\eta t_*^\beta)\bigg)^p}\right]\\
		&\leq \left({m_\varepsilon}E_\beta(-\eta t_*^\beta)\right)^p \left[\frac{f_{i_*}(v)}{v_{i_*}}+\frac{\eta}{m^{p-1}}\sup_{t\geq 1}\frac{t^{\beta-\alpha_{i_*}}E_{\beta,1-\alpha_{i_*}+\beta}(-\eta t^\beta)}{\bigg(E_\beta(-\eta t^\beta)\bigg)^p}\right].
	\end{align*}
	Taking $\beta=\underline{\alpha}/p$, by \eqref{addf}, we see that
	\[
	^{\!C}D^{{\alpha}_{i_*}}_{0^+}z_{i_*}(t_*)<0.
	\]
However, from \eqref{ctra2-2}, it deduces that $^{\!C}D^{\alpha_{i_*}}_{0^+}z_{i_*}(t^*)\geq 0,$ a contradiction. Hence, we conclude that $z_i(t)< 0$ for all $t\geq 0$ and $i=1,\dots,d.$ That is,	
	$$0\le \frac{\Phi_i(t,\omega)}{v_i}< {m_\varepsilon} E_{\underline{\alpha}/p}(-\eta t^{\underline{\alpha}/p}),\;\forall t\ge 0,\ i=1,\dots,d.$$
	From this, by letting $\varepsilon\to 0$, then
 $$\Phi_i(t,\omega)\leq {\frac{m}{E_{\underline{\alpha}/p}(-\eta)}}  v_iE_{\underline{\alpha}/p}(-\eta t^{\underline{\alpha}/p}),\;\forall t\ge 0,\ i=1,\dots,d.$$
 The proof is complete.
	\end{proof}
 \begin{remark}
     Define
	\begin{equation*}
	I(\eta):=\sup_{t\geq 1}\frac{t^{\underline{\alpha}/p-\alpha_i}E_{\underline{\alpha}/p,1+\underline{\alpha}/p-\alpha_i}(-\eta t^{\underline{\alpha}/p})}{\bigg(E_{\underline{\alpha}/p}(-\eta t^{\underline{\alpha}/p})\bigg)^p}.
	\end{equation*}
 We consider the following two cases.
 
\noindent {\bf Case I}: $p=1$. In this case, we obtain the estimates
\begin{align}
I(\eta)&\leq \sup_{t\geq 1}\frac{E_{\underline{\alpha},1+\underline{\alpha}-\alpha_i}(-\eta t^{\underline{\alpha}})}{E_{\underline{\alpha}}(-\eta t^{\underline{\alpha}})}\notag\\
&=\sup_{u\geq \eta}\frac{E_{\underline{\alpha},1+\underline{\alpha}-\alpha_i}(-u)}{E_{\underline{\alpha}}(-u)}\notag\\
&\leq \sup_{u\geq 0}\frac{E_{\underline{\alpha},1+\underline{\alpha}-\alpha_i}(-u)}{E_{\underline{\alpha}}(-u)},\label{eq4_r}
\end{align}
here $u:=\eta t^{\underline{\alpha}}$. Notice that the quantity $\displaystyle\sup_{u\geq 0}\frac{E_{\underline{\alpha},1+\underline{\alpha}-\alpha_i}(-u)}{E_{\underline{\alpha}}(-u)}$ is finite and does not depend on $\eta$. From \eqref{eq4_r}, we have
\[
0\leq \eta I(\eta)\to 0\;\;\text{as}\;\;\eta\to 0,
\]
which together the assumption $\displaystyle\frac{f_i(v)}{v_i}<0$ implies that there exists an $\eta>0$ small enough such that
\begin{equation*}
	\frac{f_{i}(v)}{v_{i}}+\eta\sup_{t\geq 1}\frac{t^{\underline{\alpha}-\alpha_i}E_{\underline{\alpha},1+\underline{\alpha}-\alpha_i}(-\eta t^{\underline{\alpha}})}{E_{\underline{\alpha}}(-\eta t^{\underline{\alpha}})}<0.
	\end{equation*}
\noindent {\bf Case II}: $p>1$. In this case, for $t\geq 1$, then
\begin{align*}
\frac{t^{\underline{\alpha}/p-\alpha_i}E_{\underline{\alpha}/p,1+\underline{\alpha}/p-\alpha_i}(-\eta t^{\underline{\alpha}/p})}{E_{\underline{\alpha}/p}(-\eta t^{\underline{\alpha}/p})^p}&=\frac{\eta^{p-1}t^{\underline{\alpha}-\alpha_i}\eta t^{\underline{\alpha}/p}E_{\underline{\alpha}/p,1+\underline{\alpha}/p-\alpha_i}(-\eta t^{\underline{\alpha}/p})}{(\eta t^{\underline{\alpha}/p}E_{\underline{\alpha}/p}(-\eta t^{\underline{\alpha}/p}))^p}\\
&\leq \frac{\eta^{p-1}\eta t^{\underline{\alpha}/p}E_{\underline{\alpha}/p,1+\underline{\alpha}/p-\alpha_i}(-\eta t^{\underline{\alpha}/p})}{(\eta t^{\underline{\alpha}/p}E_{\underline{\alpha}/p}(-\eta t^{\underline{\alpha}/p}))^p}\\
&=\eta^{p-1}\frac{uE_{\underline{\alpha}/p,1+\underline{\alpha}/p-\alpha_i}(-u)}{(uE_{\underline{\alpha}/p}(-u))^p},
\end{align*}
where $u:=\eta t^{\underline{\alpha}/p}$. Thus, for $\eta\in (0,1]$, we have
\begin{align*}
I(\eta)&\leq \sup_{u\geq \eta}\eta^{p-1}\frac{uE_{\underline{\alpha}/p,1+\underline{\alpha}/p-\alpha_i}(-u)}{(uE_{\underline{\alpha}/p}(-u))^p}\\
&\leq \eta^{p-1}\max\Big\{\sup_{u\in [\eta,1]}\frac{uE_{\underline{\alpha}/p,1+\underline{\alpha}/p-\alpha_i}(-u)}{(uE_{\underline{\alpha}/p}(-u))^p},\sup_{u\geq 1}\frac{uE_{\underline{\alpha}/p,1+\underline{\alpha}/p-\alpha_i}(-u)}{(uE_{\underline{\alpha}/p}(-u))^p}\Big\}.
\end{align*}
We see that
\begin{align*}
\sup_{u\in [\eta,1]}\frac{uE_{\underline{\alpha}/p,1+\underline{\alpha}/p-\alpha_i}(-u)}{(uE_{\underline{\alpha}/p}(-u))^p}< \frac{1}{\eta^{p-1}}\times \frac{1}{E_{\underline{\alpha}/p}(-1)^p}.
\end{align*}
Furthermore, it is not difficult to check that the limit
\[
\lim_{u\to\infty}\frac{uE_{\underline{\alpha}/p,1+\underline{\alpha}/p-\alpha_i}(-u)}{(uE_{\underline{\alpha}/p}(-u))^p}
\]
exists and is finite. Due to the fact that $E_{\underline{\alpha}/p,1+\underline{\alpha}/p-\alpha_i}(- t^{\underline{\alpha}/p})$, $E_{\underline{\alpha}/p}(- t^{\underline{\alpha}/p})$ are continuous and positive on $[0,\infty)$, the quantity
\begin{equation*}
	\sup_{u\geq 1}\frac{uE_{\underline{\alpha}/p,1+\underline{\alpha}/p-\alpha_i}(-u)}{(uE_{\underline{\alpha}/p}(-u))^p}
	\end{equation*}
	is finite and does not depend on $\eta$. In short, we also obtain $0<\eta I(\eta)\to 0$ as $\eta\to 0$.
	Using the assumption $\displaystyle\frac{f_i(v)}{v_i}<0$, there is an $\eta>0$ small enough such that
	\begin{equation*}
	\frac{f_{i}(v)}{v_{i}}+\frac{\eta}{m^{p-1}}\sup_{t\geq 1}\frac{t^{\underline{\alpha}/p-\alpha_i}E_{\underline{\alpha}/p,1+\underline{\alpha}/p-\alpha_i}(-\eta t^{\underline{\alpha}/p})}{E_{\underline{\alpha}/p}(-\eta t^{\underline{\alpha}/p})^p}<0.
	\end{equation*}
The statement \eqref{addf} is completely clarified.
 \end{remark}
	\begin{remark}
	With an arbitrary initial condition $\omega\in\R^d_{\geq 0}$, consider the system \eqref{Eq main} when $\alpha_1=\dots=\alpha_d=\alpha$ and $p=1$. By choosing $\eta>0$ such that
	\begin{equation*}
	\frac{f_i(v)}{v_i}+\eta<0,\; i=1,\dots,d,
	\end{equation*}
	we obtain a sharp estimate for the solution $\Phi(\cdot,\omega)$ as
	\[
	\Phi_i(t,\omega)\leq \frac{\|\omega\|_v}{E_{\alpha}(-\eta)}v_iE_\alpha(-\eta t^\alpha),\quad\forall t\geq 0,\;i=1,\dots,d.
	\]
	\end{remark}
	\begin{remark}
	Consider the system \eqref{Eq main} when $\alpha_1=\dots=\alpha_d=\alpha$ and $p>1$. Then, the condition \eqref{addf} becomes
	\begin{equation}\label{eq_add_r}
	\frac{f_i(v)}{v_i}+\frac{\eta}{m^{p-1}}\sup_{t\geq 1}\frac{t^{\alpha/p-\alpha}E_{\alpha/p,1+\alpha/p-\alpha}(-\eta t^{\alpha/p})}{E_{\alpha/p}(- \eta t^{\alpha/p})^p}<0,\; \forall i=1,\dots,d.
	\end{equation}
	In this case, the optimal estimate for the solution $\Phi(\cdot,\omega)$ is
	\[
	\Phi_i(t,\omega)\leq  \frac{m}{E_{\alpha/p}(-\eta)}v_i E_{\alpha/p}(-\eta t^{\alpha/p}),\;\forall t\ge 0,\ i=1,\dots,d,
	\]
 where $\eta>0$ is small enough satisfying \eqref{eq_add_r} and $m=\|\omega\|_v$.
	\end{remark}
Before closing this part, we introduce an application of the main result in our current work concerning the asymptotic behaviour of solutions to a class of fractional order systems modelling $d$ cooperating biological species. Let the particular class of fractional-order Kolmogorov systems
\begin{equation} \label{Eq coll}
		\begin{cases}
			^{\!C}D^{\hat{\alpha}}_{0^+}w(t)&={\text{diag}{(w(t))}}(b+f(w(t))),\ \forall t> 0,\\
			w(0)&=\omega \in \R_{\geq 0}^d,
		\end{cases}
	\end{equation}
here $\hat{\alpha}\in (0,1]^d$, $b\in\R^d$ and $f:\R^d\rightarrow \R^d$ is continuous. 
When $\alpha_1=\cdots=\alpha_d=1$, this is a model of Lotka-Volterra systems (a subclass of Kolmogorov systems) which has been extensively studied in the literature (see e.g., \cite{Smith86, Aeyels1}). For the case $\alpha_1=\cdots=\alpha_d\in (0,1)$, the stability of the equilibrium point of some of its special forms was reported in \cite{Ahmed, Elsadany, Tuan16}. Suggested by Theorem \ref{Rate-p>1}, we propose the following corollary.
\begin{corollary}\label{lks}
If $b\in \R^d_+$ and $f(\cdot)$ satisfies the assumptions  $\textup{(A1)}-\textup{(A3)}$, then the system \eqref{Eq coll} has a unique equilibrium point $\omega^*\in\R^d_{+}$. Furthermore, it is globally attractive, that is, for any initial condition $\omega\in \R^d_{+}$, we have
\[
\lim_{t\to\infty}\Phi(t,\omega)=\omega^*.
\]
Furthermore, the convergence rate of solutions does not exceed $t^{-\underline{\alpha}/p}$ as $t\to\infty$.
\end{corollary}
\begin{proof}
The proof is obtained by combining the arguments as in the proof of Proposition \ref{boundsol}, Theorem \ref{Rate-p>1} and ideas proposed by H.L. Smith \cite[Theorem 2.1]{Smith86} and by P.L. Leenheer and D. Aeyels \cite[Theorem 5]{Aeyels1}.
\end{proof}
\section{Numerical examples} 
This section presents some numerical examples to illustrate the given theoretical results.
\begin{example}
	Consider the system
	\begin{equation} \label{Eq main1}
		\begin{cases}
			^{\!C}D^{\hat{\alpha}}_{0^+}w(t)&=f(w(t)),\ \forall t> 0,\\
			w(0)&=\omega \in \R_{\geq 0}^2,
		\end{cases}
	\end{equation}
here
$$\hat{\alpha}=\left(\begin{array}{cc}
	0.24 \\ 0.55 \end{array}\right),\ {f(w_1,w_2)=\left(\begin{array}{cc}
	-3\sqrt{w_1^3}+2w_1\sqrt{w_2} \\ \sqrt{w_1}w_2-4\sqrt{w_2^3} \end{array}\right).}$$
It is clear to see that the function $f(\cdot)$ is cooperative and homogeneous of degree $p=\displaystyle\frac{3}{2}$. Due to $f(1,1)\prec 0$, we conclude based on  Theorem \ref{Rate-p>1} that the system is globally attractive.  
\begin{figure}
	\begin{center}
		\includegraphics[scale=.7]{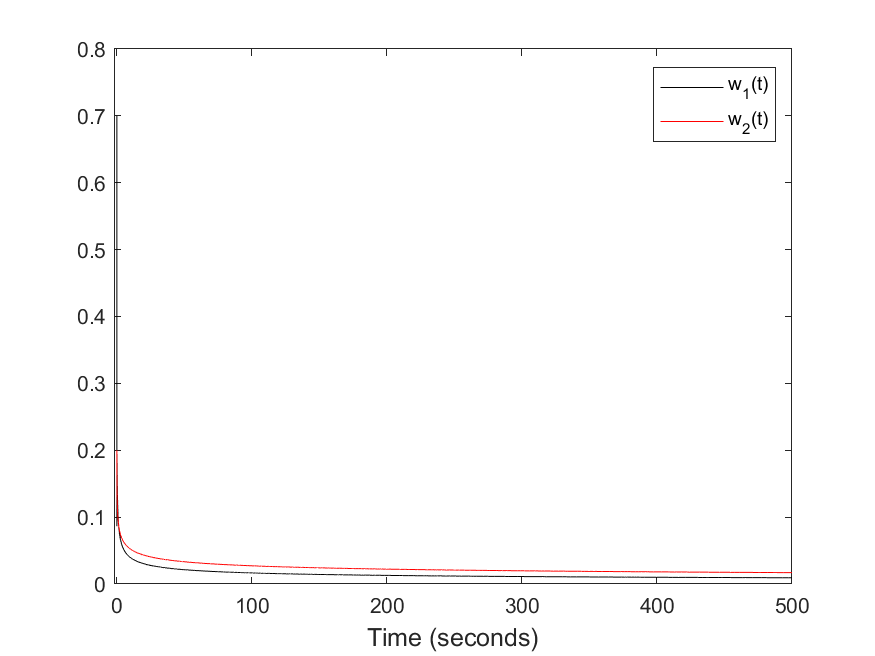}
	\end{center}
	\begin{center}
		\caption{Orbits of the solution to the system \eqref{Eq main1} with the initial condition $ \omega=(
	0.7,0.2)^{\rm T}$.}
	\end{center}
\end{figure}
\end{example}
\begin{example}
	Consider the system
	\begin{equation} \label{Eq main2}
		\begin{cases}
			^{\!C}D^{\hat{\alpha}}_{0^+}w(t)&=f(w(t)),\ \forall t> 0,\\
			w(0)&=\omega \in \R_{\geq 0}^3,
		\end{cases}
	\end{equation}
	with
	$$\hat{\alpha}=\left(\begin{array}{cc}
		0.45 \\ 0.45 \\ 0.45 \end{array}\right),\ {f(w_1,w_2,w_3)=\left(\begin{array}{cc}
		-w_1+w_2+w_3 \\ \sqrt{w_1^2+w_3^2}-4w_2\\ w_1+\sqrt{w_2^2+w_3^2}-5w_3 \end{array}\right).}$$
In this case, it is easy to check that the function $f(\cdot)$ is cooperative and homogeneous of degree $p=1$. Since $f(3,1,1)\prec 0$, by Theorem \ref{Rate-p>1}, every nontrivial solution to the system converges to the origin. 
	\begin{figure}
		\begin{center}
			\includegraphics[scale=.7]{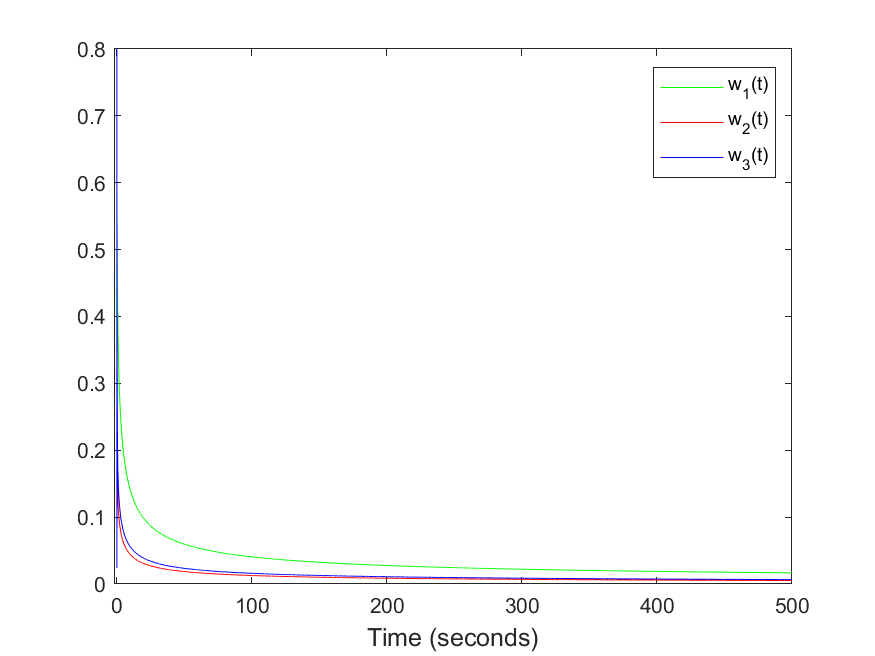}
		\end{center}
		\begin{center}
			\caption{Orbits of the solution to the system \eqref{Eq main2} with the initial condition $\omega=(
		0.5, 0.3, 0.8)^{\rm T}.$}
		\end{center}
	\end{figure}
\end{example}
\begin{example}
Consider a fractional-order two-dimensional Lotka--Volterra system
\begin{equation} \label{Eq main3}
		\begin{cases}
			^{\!C}D^{\hat{\alpha}}_{0^+}w(t)&=\text{diag}{(w(t))}(b+f(w(t))),\ \forall t> 0,\\
			w(0)&=\omega \in \R_{\geq 0}^2,
		\end{cases}
	\end{equation}
here
$$\hat{\alpha}=\left(\begin{array}{cc}
	0.4 \\ 0.6 \end{array}\right),\ b=(1,1)^{\rm T},\,{f(w_1,w_2)=\left(\begin{array}{cc}
	-3w_1+w_2 \\ w_1-w_2 \end{array}\right)}.$$
The system \eqref{Eq main3} has a unique nontrivial equilibrium point as $(1,2)^{\rm T}$. By Corollary \ref{lks}, we claim that for any $\omega\in \R^2_{+}$, the unique solution $\Phi(\cdot,\omega)$ satisfies
$\lim_{t\to\infty}\Phi(t,\omega)=\left(\begin{array}{cc}
	1 \\ 2 \end{array}\right).$
\begin{figure}
		\begin{center}
			\includegraphics[scale=.7]{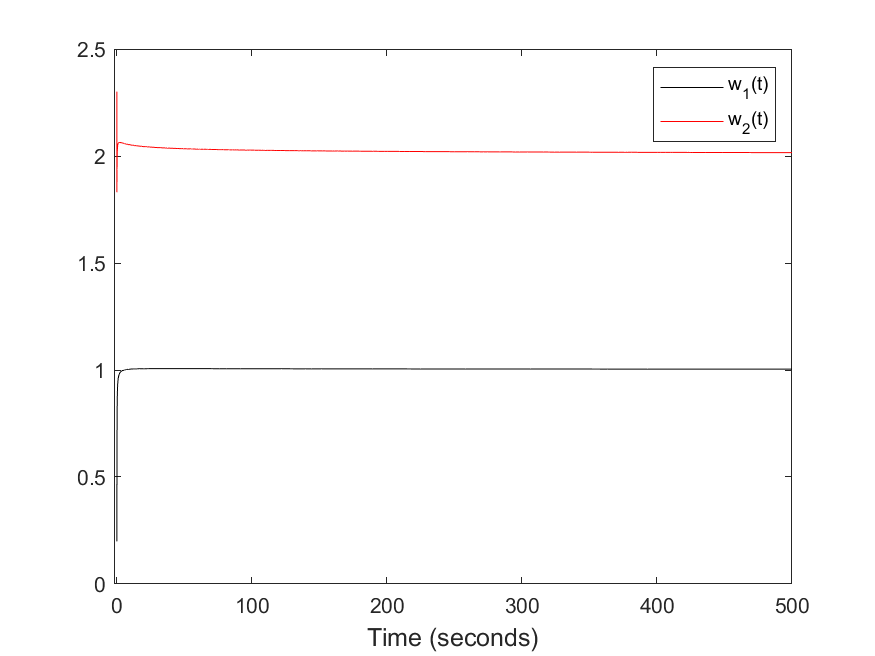}
		\end{center}
		\begin{center}
			\caption{Orbits of the solution to the system \eqref{Eq main3} with the initial condition $\omega=(0.2, 2.3)^{\rm T}.$}
		\end{center}
	\end{figure}
\end{example}

\end{document}